\documentclass{mas}

\usepackage[T1]{fontenc}
\usepackage[utf8]{inputenc}
\usepackage[francais,english]{babel}
\usepackage{xcolor}
\usepackage[normalem]{ulem}
\usepackage{soul}

\newcommand{\lp }{\left(}
\newcommand{\rp  }{\right)}

\newcommand{\dtr }{\rm det}

\newcommand{\R}{\mathbb{R}}
\newcommand{\C}{\mathbb{C}}

\newcommand{\e}{{\rm e}}
\newcommand{\pr}{\partial}

\newcommand{\tr}{{\rm Tr}}
\newcommand{\sgf}{\sigma_f^2}
\newcommand{\sgfp}{\sigma_{f,P}^2}

\newcommand{\smfrac}[2]{{\textstyle \frac{#1}{#2}}}
\def\<{\langle}
\def\>{\rangle}

\begin{document}
\markboth{HA,CO} {pol}
\title{Some Remarks on Preconditioning Molecular Dynamics}

\author{HOUSSAM ALRACHID}
\address{IDP Laboratory, Orl\'eans University, Orl\'eans, France\\
houssam.alrachid@univ-orleans.fr}

\author{LETIF MONES}
\address{Mathematics Institute, University of Warwick, Coventry, United Kingdom\\
l.mones@warwick.ac.uk}

\author{CHRISTOPH ORTNER}
\address{Mathematics Institute, University of Warwick, Coventry, United Kingdom\\
c.ortner@warwick.ac.uk}

\maketitle

\begin{abstract}
We consider a Preconditioned Overdamped Langevin algorithm that does not alter the invariant distribution (up to controllable discretisation errors) and ask whether preconditioning improves the standard model in terms of reducing the asymptotic variance and of accelerating convergence to equilibrium. We present a detailed study of the dependence of the asymptotic variance on preconditioning in some elementary toy models related to molecular simulation. Our theoretical results are supported by numerical simulations.
\end{abstract}


 \section{Introduction}\label{intro}

The problem of convergence to equilibrium for diffusion processes has attracted considerable attention in recent years. Due to the significant computational cost associated with MCMC type algorithms it is important to understand, and where possible accelerate, the convergence to equilibrium of systems in statistical physics, materials science, biochemistry, machine learning and many other areas.

In such applications it is typically necessary to compute expectations of the form
\begin{equation} \label{eq:intro:muf}
   \mu(f)= \int_{\R^N} f(x)\mu(dx)
\end{equation}
of an observable $f$ with respect to a target probability distribution $\mu(dx)$ on $\R^N$,
where $\mu$ is of the form
\begin{equation*}\label{gib}
   \mu(dx)=Z_{\mu}^{-1}\e^{-E_N}dx,
\end{equation*}
where $E_N:\R^N \rightarrow \R$ is e.g. a potential energy and $Z_{\mu}^{-1}$ the normalisation constant. Note that we have absorbed temperature into $E_N$. The usual difficulty is that the integration
\eqref{eq:intro:muf} cannot be performed directly due to the high dimensionality
of the problem and MCMC methods are instead employed.


Ill-conditioning of $E_N$, which can be induced by a variety of mechanisms, but
in particular high-dimensionality (large $N$), is a common challenge to 
overcome in order to construct an efficient sampling scheme. An
attractive approach is to {\em precondition} the
MCMC algorithm. The algorithm is transformed by a well-chosen operator
(the preconditioner) in a way that does not alter the invariant measure
but (hopefully) accelerates convergence.

The purpose of this paper is to explore to what extent (or, whether at all)
preconditioning of a Langevin-type algorithm helps to accelerate the computation
of expectations. Our study is motivated by recent advances such as the
Riemannian Manifold MALA [\refcite{giro09}], Stochastic Newton Methods
[\refcite{mart12}], non-reversible diffusions [\refcite{dunc,Lelievre}], optimal
scaling for Langevin algorithms [\refcite{bes09,Pillai}] and affine-invariant
sampler [\refcite{hou}].

All of these references present different variants of preconditioning or
related modifications of MCMC methods and result in an improved performance. Reviewing all these different approaches would go beyond the scope
of this introduction, however, we mention the Riemannian Manifold Monte Carlo 
method  [\refcite{giro09}, \refcite{xifa}] as the main motivation for the present
study. In this method, preconditioning is understood as performing
MCMC on a Riemannian manifold defined via a ``metric'' $P(x)$, the
preconditioner. A range of preconditioners, ideally the hessian $\nabla^2
E_N(x)$ or a positive definite matrix $P(x)$ closely approximating the Hessian,
are tested. In a broad range of examples (including, e.g., logistic regression,
a stochastic volatility model, and an ODE inference example) it has been shown in
numerical tests that a well-chosen preconditioner improves both the mixing time
as well as the convergence of the probability density functions to the target
measure, with speed-ups ranging from moderate $O(1)$ factors to orders of
magnitude depending on the application.

Motivated by these promising results we applied analogous preconditioned
sampling algorithms  to some model molecular systems, but did not always (rarely
in fact) observe the speed-ups we expected. 
For instance, in Figure~\ref{fig:pmf_mtd_mtd}
we show the convergence of reconstructed free energy profiles from
metadynamics [\refcite{Laio02}] simulations comparing the unpreconditioned and preconditioned dynamics using a Hessian--based preconditioner with varying parameter sets. This type of preconditioner usually yields at least an order of magnitude speed-up in typical geometry optimisation
tasks, but fails to accelerate the assembly of the free energy surface (see Appendix~\ref{eq:mtd}
for the precise setup of the test).
The question naturally arises whether this failure is due to a lack of
fine-tuning or due to a more fundamental limitation.

\begin{figure}
   \centering
   \includegraphics[width=0.9\columnwidth]{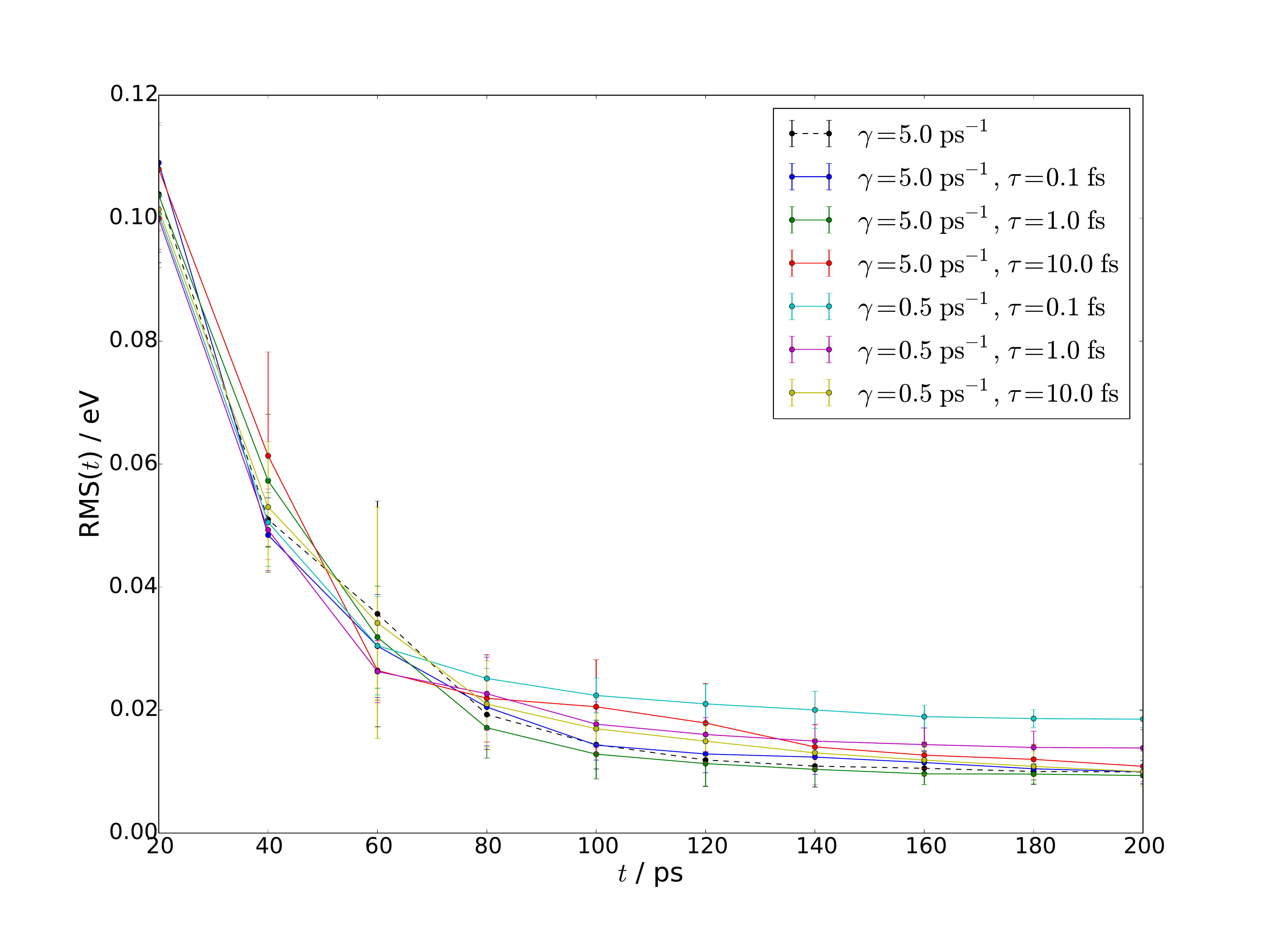}
   \caption{RMS errors of the reconstructed free energy surface profiles from metadynamics simulations without (black dashed line) and with preconditioning (coloured solid lines). Error bars represent one standard deviation based on 10 independent simulations. See \ref{sec:mtd} for the detailed setup.}
   \label{fig:pmf_mtd_mtd}
\end{figure}

Thus, to understand better these observations we will study some elementary analytical
and numerical examples, which capture some essential characteristics of typical
molecular systems, but where explicit results can still be obtained. The origin
of the difficulty comes from the fact that it is highly dependent not only on
$\mu$ but also on the observable $f$ whether preconditioning can achieve a
significant (or, any) speed-up. We will demonstrate that for some typical
observables $f$, even moderate preconditioning can achieve significant speed-ups
while for other, equally common, observables no speed-up should be at all
expected even if $E_N$ is highly ill-conditioned. Moreover, we will show that
the dimensionality (1D, 2D, 3D) of the molecular structure plays a crucial role.

While our discussion is primarily motivated by applications in molecular simulations,
it should straightforwardly adapt the arguments and findings to other
application areas.

\subsection{Langevin Algorithm}
The most commonly employed algorithms in molecular simulation are based on
discretising the Langevin equation, but for the sake of simplicity we will
focus on the overdamped Langevin equation,
\begin{equation} \label{over}
   dX_t= - \nabla E_N(X_t) dt + \sqrt{2} dW_t,
\end{equation}
where $W_t$ is a $N$-dimensional standard Brownian motion. Under mild technical conditions on $E_N$ and $X_t$ it is known that the
dynamics $(X_t)_{t\geq 0}$ is ergodic  with respect to the measure $\mu$ (see
[\refcite{rob96}]).

Discretising in time,
\begin{equation} \label{eq:euler}
      X_{m+1} = X_m - \delta \nabla E_N(X_m) + \sqrt{2\delta} R_m,
\end{equation}
where $R_m \sim N(0, I_{N \times N})$ and $\delta$ is a parameter quantifying the
size of the discrete time increment (time-step), we obtain a Markov chain with
invariant measure $\tilde\mu$, where typically $\tilde\mu-\mu \sim  O(t^{-1/2}) +  O(\delta)$. Here, $O(t^{-1/2})$ represents the statistical error due to the finite length of the simulation, while $O(\delta)$ represents the bias due to
time discretisation.

Adding metropolisation to \eqref{eq:euler} leads to the Metropolis Adjusted Langevin Algorithm (MALA). It is
well known that the measure $\mu$ is invariant for the MALA
[\refcite{rob04,rob96}]. In molecular simulation it is
common to assume that the bias in $\tilde\mu$ is negligible compared to
modelling and statistical errors and therefore we
will not consider metropolisation in the present paper.
The resulting algorithm is called the Unadjusted Langevin Algorithm (ULA).

As a matter of fact, we will argue in Section \ref{sec:time-discretisation} that the issues we
are addressing are largely unrelated to the time-discretisation, hence we will
focus on the continuous process \eqref{over} instead of the
Markov chain \eqref{eq:euler}.

\def\emp{\epsilon}

Thus, given an observable $f\in L^1(\mu)$, we are interested in quantifying the convergence
\begin{equation} \label{ergo}
   \emp_T(f) := \frac{1}{T}\int_0^T f(X_t) dt \rightarrow \mu(f)
     \qquad  \text{for $\mu$-a.e } X_0.
\end{equation}
Under additional assumptions on $\mu$ and $f$, this convergence result is
accompanied by a central limit theorem which characterises the asymptotic
distribution of the fluctuations, i.e.
\begin{equation}
   \sqrt{t} \lp\emp_t(f)-\mu(f)\rp \xrightarrow{D} \mathcal{N}(0,\sgf),
\end{equation}
where  $\sgf$ is known as the asymptotic variance for the observable $f$ (see [\refcite{catt,kipnis}]).

In Section \ref{clte} we will explicitly compute $\sgf$ for some simple
energy functionals and observables, which mimic typical objects
of interest in molecular simulations and demonstrate that for some typical
observables, $\sigma_f$ may be strongly dependent on the conditioning of $E_N$,
in particular on system size $N$, while for others the dependence on $N$
is negligible.

\subsection{Preconditioned Langevin Algorithm}
For a (fixed) preconditioner $P \in \R^{N \times N}$, symmetric positive definite, we consider the preconditioned
Overdamped Langevin dynamics ($P$-Langevin),
\begin{equation}\label{pol}
   dX_t^P = -  P^{-1}\nabla E_N(X_t^P)dt + \sqrt{2} P^{-1/2} dW_t.
\end{equation}
The standard overdamped Langevin dynamics is recovered by taking $P=I$. It is in
principle possible to allow $P = P(X)$ but for simplicity we will not consider
this in the present work. Note that, if $P$ is fixed then the coordinate
transform $Z = P^{1/2} X$ allows us to easily lift results from \eqref{over} to
\eqref{pol}; see Section \ref{sec:coordinate-transform} for more details.

The preconditioner does not affect the invariance property of the diffusion
process, i.e., the target measure $\mu$ is still invariant for the P-Langevin
process \eqref{pol}. However, it can affect the convergence of the process to
the invariant measure. That is, for $X_t^P \sim \mu_t^P = \varphi_t^P dx$, we
will characterise in Section~\ref{convin} the rate of convergence of $\varphi_t^P$ to $\varphi_\infty$ ($\varphi_\infty $ denotes the density of $\mu$) and demonstrate how
preconditioning improves this rate.

Moreover, we also obtain
\begin{equation} \label{ergo2}
   \emp_T^P(f) := \frac{1}{T}\int_0^T f(X_t^P)dt
   \to \mu(f), \qquad  \text{for $\mu$-a.e } X_0.
\end{equation}
and analogously to the standard Langevin dynamics, a central limit theorem
characterizes the asymptotic distribution of the fluctuations,
\begin{equation}\label{tcl2}
  \sqrt{t} \lp\emp_t^P(f) - \mu(f) \rp \xrightarrow{D} \mathcal{N}(0,\sgfp),
\end{equation}
where $\sgfp$ is the asymptotic variance of $f$
under $P$-Langevin dynamics.

The main aim of our paper is to present several simplified but
still realistic examples at which we can observe {\em whether or not} preconditioning
accelerates sampling in the sense that it achieves a reduction in the asymptotic
variance, i.e., $\sgfp \ll \sgf$.

\subsection*{Outline}

The rest of the paper is organised as follows. In Section~\ref{sec:preliminaries}
we present the model Hamiltonians, $E_N$, to motivate some key assumptions
that we make throughout our analysis. We also recall the coordinate transform that
we use to reduce $P$-Langevin to standard Langevin, and we explain why the time-discretisation
is a negligible component and can therefore be ignored for our analysis.

In Section~\ref{convin} we show a long time convergence result to the invariant measure in the $P$-Langevin process. For some quadratic model Hamiltonians we can then precisely quantify the speed-up afforded through preconditioning.

Section~\ref{clte} is devoted to describe how the central limit theorem \eqref{tcl2} arises from the solution of the Poisson equation associated with the generator of the dynamics. This is then followed by a detailed analysis of $\sgf$ and $\sgfp$ for some quadratic model Hamiltonians and observables.

In Section~\ref{test} we show a numerical application to illustrate the reduction in the asymptotic variance.

\section{Preliminaries}\label{preli}
\label{sec:preliminaries}

\subsection{Coordinate transformation}
\label{sec:coordinate-transform}
As mentioned above, a convenience afforded by our assumption that the
preconditioner $P$ is constant, is that a simple coordinate transformation
can transform the P-Langevin dynamics \eqref{pol} into standard Langevin
dynamics \eqref{over} by taking $Z_t = P^{1/2} X_t$. We will now briefly review this transformation.

First, let us introduce the new coordinates and associated energy
$$
   z := P^{1/2} x \qquad \text{and} \qquad
   E_N^P(z) := E_N(x) = E_N(P^{-1/2} z).
$$
The standard overdamped Langevin dynamics for $z$ then reads
\begin{equation} \label{pmala2}
   \begin{split}
   dZ_t &= - \nabla_z E_N(P^{-1/2} Z_t) dt + \sqrt{2} dW_t \\
         &= - P^{-1/2} \nabla_x E_N(P^{-1/2} Z_t) + \sqrt{2} dW_t.
   \end{split}
\end{equation}
Upon multiplying the equation with $P^{-1/2}$ we clearly recover
\eqref{pol}; that is, \eqref{pmala2} and \eqref{pol} are indeed equivalent. In the new coordinate system the infinitesimal generator operator of the diffusion process $Z_t$ is given by:
\begin{equation}\label{ops}
\mathcal{L}_P= - \nabla_z E_N^P(z)\cdot\nabla_z +  \Delta_z.
\end{equation}

In terms of estimating observables, we obtain that
\begin{equation} \label{empTP-intermsof-Z}
   \emp_T^P(f) = \frac{1}{T}\int_0^T f(X_t^P)dt
   = \frac{1}{T} \int_0^T f(P^{-1/2} Z_t) dt.
\end{equation}

The key observation in analysing how preconditioning changes the properties of
$E_N^P$ and hence of the Langevin dynamics is that
$$
   \nabla_z^2 E_N^P(z) = P^{-1/2} \nabla_x^2 E_N(x) P^{-1/2}.
$$
In the following, if $H$ and $P$ are symmetric positive definite, we will write
$H_P := P^{-1/2} H P^{-1/2}$. We also observe, for future reference, that the spectrum of $H_P$ satisfies 
$$
   \inf \sigma(H_P) = \inf_{v^T P v = 1} v^T H v \qquad \text{and} \qquad
   \sup \sigma(H_P) = \sup_{v^T Pv = 1} v^T H v.
$$

%

\subsection{Model Hamiltonians}
\label{sec:model_hamiltonians}
Let $x \in \R^N$ describe a system of $N/d$ particles at
positions $(y_\alpha)_{\alpha = 1}^{N/d} \subset \R^d$, then
a simple model for potential energy is given by
$$
   E_N(y) = \sum_{\alpha = 1}^{N/d} \sum_{\beta \neq \alpha}
      \phi( |y_\alpha - y_\beta| ),
$$
where $\phi$ is e.g. a Lennard-Jones type potential,
$\phi(r) = r^{-12} - 2 r^{-6}$.
Such systems exhibit complex meta-stable behaviour, which is an issue to
be entirely separated from the ill-conditioning due to high dimension.

A much simpler situation is a mass-spring model, where $u = (u_\alpha) \in
\R^N$, with $u_\alpha$ e.g. denoting out-of-plane displacement of a particle,
and particle connectivity described by an equivalence relation $\alpha \sim
\beta$. Then the energy can be written as
\begin{equation} \label{eq:nonlinear_model}
E_N(u) = \sum_{\alpha \sim \beta} \phi_{\alpha\beta}( u_\alpha - u_\beta ),
\end{equation}
where $\phi_{\alpha\beta}$ could be taken as strictly convex to avoid
meta-stability. We will use systems of this kind in our numerical
experiments.

In order to admit explicit analytical calculations we simplify
further $E_N$ by expanding it about an equilibrium (e.g., at $u = 0$) which then
yields a quadratic energy
\begin{equation} \label{eq:quad_model}
   E_N(u) = \sum_{\alpha \sim \beta} \smfrac{k_{\alpha\beta}}{2}
         |u_\alpha - u_\beta|^2 =: {\textstyle \frac12} u^T H u.
\end{equation}
The spring constants $k_{\alpha\beta}$ could model how the interaction
between different atomic species / environments differs. We assume throughout that there
exist bounds $\underline{k}, \bar{k}$ on the spring constants and
a bound $\bar{n}$ on the number of neighbours, which are both independent of $N$.
More precisely, we assume that
\begin{equation} \label{eq:bounds_quad}
   0 < \underline{k} \leq k_{\alpha\beta} \leq \bar{k} \qquad
   \text{and} \qquad
   \# \{ \beta : \alpha \sim \beta \} \leq \bar{n}.
\end{equation}

Note that we have chosen a scaling of the energies where increasing system size $N \to
\infty$ does not yield a continuum limit but rather an infinite lattice
system. The consequence is that if we have a bound in the spring constants
$0 < k_{\alpha\beta} \leq \bar{k}$, and each atom $\alpha$ is connected with
at most $\bar{n}$ neighbours, then it follows readily that
\begin{equation} \label{eq:bound_hessian}
   |E_N(u)| \lesssim \bar{n} \bar{k} |u|^2 \qquad \text{or, equivalently}
      \qquad \| H \| \lesssim \bar{n} \bar{k},
\end{equation}
where $\|H\|$ denotes the $\ell^2 \to \ell^2$ operator norm.

To be even more specific let us assume that $E_N$ is a $d$-dimensional lattice
model, i.e., each atom index $\alpha$ corresponds to a coordinate $\ell_\alpha
\in \{0, \dots, M+1\}^d$ and $\alpha \sim \beta$ if and only if $|\ell_\alpha -
\ell_\beta| = 1$, that is
\begin{equation} \label{eq:quad-lattice-model}
    E_N(u) := \sum_{\alpha_1=0}^M \cdots \sum_{\alpha_d = 0}^M
          \sum_{j = 1}^d k_{\alpha,\alpha+e_j} \big|u(\alpha+e_j) - u(\alpha)\big|^2.
\end{equation}
 By clamping the boundary sites at $u = 0$, we obtain $N = M^d$
free lattice sites. In this case, $H$ is a (possibly inhomogeneous) discrete
elliptic operator and employing \eqref{eq:bounds_quad} and
using the min-max characterisation of
eigenvalues (see, e.g., Sect XIII.1 in [\refcite{ReedSimon}])
to compare $H$ to the discrete Laplacian for which the
spectrum can be computed explicitly [\refcite{iserles}], we can readily show that there exist
constants $c_0, c_1$ such that
\begin{equation} \label{eq:spectrum-dimension}
   c_0 (j/N)^{2/d} \leq \lambda_j \leq c_1 (j/N)^{2/d},
\end{equation}
where $\sigma(H) = \{\lambda_j\,|\,j = 1,\dots,N\}$ is the ordered spectrum of $H$.
This dimension-dependence of the eigenvalue distribution will be important later on.


\subsection{A model preconditioner}
\label{sec:model_precond}
Although for the simple model problems described in
Section~\ref{sec:model_hamiltonians} it is straightforward to compute Hessians and use those
as preconditioners, this would remove us from the practice of molecular
simulations where Hessians are not normally computable. Instead, we will
consider preconditioners that only roughly capture the structure of the
energy functionals and their Hessians.

Following [\refcite{packwood}],  we will use a preconditioner of the form
\begin{equation} \label{eq:precond_nonlinear}
   v^T P v = c \sum_{\alpha \sim \beta} |v_\alpha - v_\beta|^2,
\end{equation}
where $c$ is a free parameter to be fitted to the model. The idea is that it
captures the connectivity information but not the fine details of the Hessian.
For optimisation and saddle search, preconditioners of this kind have been shown
to yield considerable speed-ups even for much more complex electronic structure
type models [\refcite{packwood}].

For example, considering $E_N$ given by \eqref{eq:quad_model}
with hessian $H := \nabla^2 E_N$ and choosing
$c:= \frac12 (\underline{k} + \bar{k})$, we observe
that
\begin{equation} \label{eq:bounds:HP}
   \smfrac{2 \underline{k}}{\bar{k}+\underline{k}} v^T P v
   \leq v^T  H  v
   \leq \smfrac{2 \bar{k}}{\bar{k}+\underline{k}} v^T P v,
\end{equation}
which in particular implies that
$$
   \kappa\big( H_P \big) \leq \frac{\bar{k}}{\underline{k}}.
$$
Here, $\kappa(A) := \|A\| \|A^{-1}\|$ denotes the condition number of
the matrix $A$.

That is, provided the inhomogeneity $\bar{k}/\underline{k}$ is not too severe,
then the preconditioned energy landscape has only very moderate conditioning,
independent of $N$, while typically $\kappa(\nabla^2 E) \to \infty$
as $N \to \infty$, with a rate depending on the connectivity. For instance,
for a lattice model \eqref{eq:spectrum-dimension} implies that
$\kappa(\nabla^2 E_N) \sim N^{2/d}$.

\subsection{Time Discretisation} \label{sec:time-discretisation}
To measure the cost/accuracy ratio of a practical sampling algorithm based on
the Langevin process we need to account also for the time-discretisation,
\begin{equation} \label{eq:pol_discrete}
   X_{n+1} = X_n - \delta P^{-1} \nabla E_N(X_n)
            + \sqrt{2 \delta} P^{-1/2} R_n,
\end{equation}
where $R_n \sim N(0, I_{N \times N})$.

It is tempting to assume that one advantage afforded by preconditioning is
to take larger time-steps. As we show in the following, this is a matter of
scale, thus justifying our choice to focus purely on the time-continuous
P-Langevin dynamics.

In the scaling that we have chosen in the model problems of Section
\ref{sec:model_hamiltonians}, both the Hessians $H$ and preconditioned Hessians $H_P$ are
bounded operators (on $\ell^2$), independent of system size $N$. Thus, the time-steps have
similar restrictions for the preconditioned and unpreconditioned Langevin
processes [\refcite{hairr}]. More precisely, in the quadratic model problem \eqref{eq:quad_model}
we may assume $\bar{k} = 1$ without loss of generality. Then, using the model
preconditioner \eqref{eq:precond_nonlinear} with $c = \bar{k} = 1$, say, we
have that $\sup \sigma(H) = \sup \sigma(H_P)$.

\begin{proposition} \label{th:time-discretisation}
   Suppose that $E_N(x) = \frac12 x^T H x$ and $P \in \R^{N \times N}$ with $H$
   and $P$ both positive definite. Suppose further that $
   \frac{\delta }{2}|P^{-1}H| < 1$. Then the invariant measure for
   \eqref{eq:pol_discrete} is Gaussian with covariance matrix
   $$C_\delta^P = \big(I - \smfrac{\delta}{2}P^{-1}H\big)^{-1}H^{-1}.$$
\end{proposition}
\begin{proof}
   See Section~\ref{prp}.
\end{proof}

It follows that the error in the covariance operator is $O(\delta)$ with
constant independence of $N$, and independent of the choice of preconditioner.
Indeed, even with an ``optimal'' preconditioner $P = H$, one would not obtain an
improvement in the bias: if there
exist eigenmodes $Hv=\lambda v$ with $\lambda \ll 1$, then
\begin{align*}
   & | C_\delta^H v - C v | = \big|(1-\smfrac\delta 2)^{-1} - 1\big| \, | C v |,  \\
   \text{whereas} \quad
   & | C_\delta^I v - Cv | = \big| (1-\smfrac{\lambda\delta}{2})^{-1} - 1\big| |Cv|,
\end{align*}
that is,
$$
   \frac{| C_\delta^H v - C v |}{| C_\delta^I v - Cv |} \sim \lambda^{-1}
      \qquad \text{for $\delta, \lambda$ sufficiently small,}
$$
where $C=H^{-1}$ is the covariance operator of the unbiased measure. Therefore, for the remainder of the paper, we will not consider the effect of
preconditioning on time-discretisation but only focus on the speed of convergence
to equilibrium in the undiscretised (P-)Langevin dynamics. We only stress that
this point of view is only valid as long as $\|H\|$ and $\|H_P\|$ are comparable.

\section{Exponential convergence to the invariant measure} \label{convin}

In this section, we prove exponential convergence to the equilibrium. For the
sake of simplicity we represent the probability density functions $\mu$ and
$\mu_t^P$ by their respective densities $\varphi_\infty$ and $\varphi_t^P$.
Under the transformation $z = P^{1/2} x$, we obtain transformed probability
densities
\begin{equation*}
   \psi_\infty (z) := (\det P)^{-1/2} \varphi_\infty(P^{-1/2} z) \quad
   \text{and} \quad
   \psi_t^P(z) := (\det P)^{-1/2} \varphi_t^P(P^{-1/2} z).
\end{equation*}
Their evolution is described by the Fokker--Planck equation
\begin{equation}\label{fpfp}
   \pr_t\psi_t^P=\mathcal{L}_P^{*}\psi_t^P
   := \nabla_z \cdot \left(\nabla_z E_N(P^{-1/2} z) \psi_t^P
         + \nabla_z \psi_t^P \right),
\end{equation}
where $\mathcal{L}_P^{*}$ is the dual operator of $\mathcal{L}_P$ (for the $L^2(dz)-$scalar product) defined in \eqref{ops}.

The proof of the following theorem is inspired from [\refcite{Lelievre}] or A.19
in [\refcite{Villani}].

\begin{theorem} \label{th:exponential_convergence}
   Suppose that $E_N\in C^2(\R^N)$, such that $\frac{1}{2}|\nabla E_N(x)|^2-\Delta E_N(x) \rightarrow +\infty $ as $|x|\rightarrow +\infty $. Then there exists
   $\lambda_P > 0$ such that for all initial conditions $\psi_0 \in L^2(\varphi_{\infty}^{-1})$, and for all times $t\geq 0$
   \begin{align*}
   \|\psi_t^P-\psi_\infty \|_{L^2({\psi_{\infty}^{-1})}}^2
   &\leq \e^{-\lambda_P   t}\|\psi_0-\psi_\infty \|_{L^2({\psi_{\infty}^{-1})}}^2,
      \qquad \text{or, equivalently,} \\
   \|\varphi_t^P-\varphi_\infty\|_{L^2({\varphi_{\infty}^{-1})}}^2
      &\leq \e^{-\lambda_P   t}\|\varphi_0-\varphi_\infty\|_{L^2({\varphi_{\infty}^{-1})}}^2,
   \end{align*}
  $||.||_{L^2(\psi_{\infty}^{-1})} $ denotes the norm in ${L^2(\R^N,\psi_{\infty}^{-1}dz)}$, namely $||f||_{L^2(\psi_{\infty}^{-1})}^2=\int_{\R^N}f^2(z)\psi_\infty ^{-1}dz $. The exponent $\lambda_P$ is the spectral gap of
   the Fokker--Planck operator $\mathcal{L}_P^*$ defined in \eqref{fpfp}
   (i.e., the smallest non-zero eigenvalue of $-\mathcal{L}_P^*$).
 \end{theorem}
\begin{proof}
   See Section~\ref{prtheo1}.
\end{proof}

We now consider the quadratic case, $E_N(x) = \frac12 x^T H x$ with
$H$ symmetric and positive definite. In this case,
we have the following characterisation of the spectrum of
$\mathcal{L}_P^*$ and hence of $\lambda_P$.  The proof is based on [\refcite{hit,Lelievre,ott}].

\begin{theorem}\label{spec}
   The spectrum of the operator $\mathcal{L}_P^*$ is
   \begin{equation}
   \sigma(\mathcal{L}_P^{*})
   = \sigma(\mathcal{L}_P)
   = \displaystyle\Big\{ - {\textstyle \sum_{\lambda\in \sigma(H_P)}} k_\lambda \lambda \,:\, k_\lambda\in \mathbb{N}\Big\}.
   \end{equation}
\end{theorem}
\begin{proof}
   See Section~\ref{prtheo2}.
\end{proof}

An immediate consequence of Theorem \ref{spec} is that
\begin{equation} \label{eq:lambdaP}
   \lambda_P = \inf \sigma(H_P) \setminus \{0\},
\end{equation}
in other words, the smallest non-zero generalised eigenvalue of
\begin{equation*}
      H v = \lambda P v.
\end{equation*}

Returning to the quadratic model Hamiltonians and model preconditioners
introduced in Sections~\ref{sec:model_hamiltonians} and~\ref{sec:model_precond},
assume that $\min \sigma(H) \sim N^{-s}$ (which is consistent with
\eqref{eq:spectrum-dimension}), while $\min\sigma(H_P) \sim 1$, then we obtain
that $\lambda_I \sim N^{-s}$ with $\lambda_P \sim 1$. This result should give
us confidence in the value of preconditioning.

 \begin{remark}
The preconditioning ideas and results presented in Sections~\ref{preli} and~\ref{convin} are similar to the Brascamp--Lieb inequality [\refcite{brasc}]. In some sense, this inequality claims that a good preconditioner is the Hessian.

Precisely [\refcite{Lelievre:17}], if $E_N$ is strictly convex, for any function $f \in L^2(\e^{-E_N})$,
$$\int \left[f - \int f \e^{-E_N} \right]^2 \e^{-E_N} \le \int \nabla f (\nabla^2 E_N)^{-1} \cdot \nabla f \e^{-E_N}, $$
where we assume here that the normalization $\int \e^{-E_N} = 1$. This means that if one considers the Fokker--Planck equation
$$\partial_t \psi = {\rm div} [ (\nabla^2 E_N)^{-1} \e^{-E_N} \nabla ( \psi \e^{E_N}) ],$$
which is associated to the following overdamped Langevin dynamics [\refcite{xifa}]:
$$dX_t = - (\nabla^2 E_N)^{-1} \nabla E_N (X_t) dt + {\rm div} [ (\nabla^2 E_N)^{-1} ] (X_t) dt + \sqrt{2 (\nabla^2 E_N)^{-1}}(X_t) dW_t, $$
then, if $E_N$ is strictly convex, we have
$$\begin{aligned}
\frac{1}{2} \frac{d}{dt} \int \left(\frac{\psi}{\e^{-E_N}} - 1\right)^2 \e^{-E_N}  & = - \int (\nabla^2 E_N)^{-1} \nabla \left( \frac{\psi}{\e^{-E_N}} \right) \cdot \nabla \left( \frac{\psi}{\e^{-E_N}}\right)  \e^{-E_N}\\
&\le
\int \left(\frac{\psi}{\e^{-E_N}} - 1\right)^2 \e^{-E_N}.
\end{aligned}$$
And thus
$$\int \left(\frac{\psi}{\e^{-E_N}} - 1\right)^2\e^{-E_N} \le \left[ \int \left(\frac{\psi_0}{\e^{-E_N}} - 1\right)^2 \e^{-E_N}\right ] \e^{-2t}, $$
whatever the potential $E_N$ is. It is in some sense ``universal''. For example, it does not depend on the temperature: if we multiply $E_N$ by a constant $\beta$ (inverse of the temperature) it remains the same, whereas understanding the dependency of the spectral gap on the temperature is tricky in general.

\end{remark}

\section{Analysis of the asymptotic variance} \label{clte}
In this section we present sufficient conditions under which the estimator
$$
   \epsilon_T^P(f)
   = \frac{1}{T} \int_0^T f(X_t^P) dt
   = \frac{1}{T}\int _0^T f(P^{-1/2}Z_t)dt
$$
satisfies a central limit theorem of the form \eqref{tcl2} and we characterise the
associated asymptotic variance.

\subsection{Generalities}

The fundamental requirements to prove the central limit theorem is establishing the well-posedness of the Poisson equation
\begin{equation}\label{poisson2}
-\mathcal{L}_P\phi(z)=f(P^{-1/2}z)-\mu(f),\,\, \mu(\phi)=0,
\end{equation}
for all bounded and continuous functions $f: \R^N\rightarrow \R $, where $\mathcal{L}_P $ is defined by \eqref{ops}, and obtaining estimates on the growth of the unique solution $\phi$. Recall that we shall assume that $\mu$ admits a smooth, strictly positive density denoted by $\psi_\infty (z)$, such that $\int_{\R^N}\psi_\infty  (z)dz = 1$ and the SDE \eqref{pmala2} has a unique strong solution.

Referring to results in [\refcite{glynn,meyn}] we suppose that the process $Z_t$  admits a Lyapunov function (see the Definition~\ref{def} in Section~\ref{prfpoi}),
which is sufficient to ensure the geometric ergodicity of $Z_t$ (see [\refcite{mattin,twe}]). In terms of the
potential energy $E_N$ and the preconditioner $P$, we require that there exists
$\beta \in (0, 1)$ such that
\begin{equation} \label{eq:coercivity-assumption}
\displaystyle\lim_{|z|\rightarrow +\infty }\inf \left[(1-\beta)|\nabla E_N^P(z)|^2+\Delta E_N^P(z)\right]>0.
\end{equation}
It is straightforward to check that this condition holds whenever $E_N$ is
strongly convex and in particular if it is of the form \eqref{eq:quad_model}.

If condition \eqref{eq:coercivity-assumption} holds, then the process $Z_t$ will be geometrically ergodic. More specifically, the law of the process $Z_t$  started from a point $z\in\R^N$ will converge exponentially fast in the total-variation norm to the equilibrium distribution $\mu$ (cf. \eqref{ergo2}).

Assuming \eqref{eq:coercivity-assumption}, we also obtain the following well-posedness result for the Poisson equation \eqref{poisson2}.

\begin{theorem} \label{th:poisson}
   Suppose that \eqref{eq:coercivity-assumption} holds, then there exists $c > 0 $,
   such that for any measurable observable $f$ satisfying  $|f|^2 \leq e^{-\beta E_N^P(z)}$, the Poisson equation \eqref{poisson2} admits a unique strong solution
 satisfying the
   bound $|\phi(z)|^2 \leq c e^{-\beta E_N^P(z)}$. In particular, $\phi(P^{1/2}
   x) \in L^2(\mu)$.
 \end{theorem}
 \begin{proof}
 See Section~\ref{prfpoi}.
 \end{proof}

The technique of using a Poisson equation to obtain a central limit theorem for
an additive functional of a Markov process is widely known (see e.g.
[\refcite{bhatt}]). For linear and quadratic observables, we can in fact produce
an analytic solution to this Poisson problem; see Section~\ref{sec:linear_observables} below.

\begin{theorem}\label{thtcl}
   Under the conditions of Theorem~\ref{th:poisson}, there exists a constant
   $0< \sgfp < \infty$ such that the asymptotic distribution of the fluctuations
   of $\emp_t^P(f)$ about $\mu(f)$ are given by the central limit theorem
\begin{equation}\label{tcltheo}
\displaystyle \sqrt{t} \lp\emp^P_t(f)-\mu(f)\rp \xrightarrow{D} \mathcal{N}(0,\sgfp),\text{ as } t\rightarrow +\infty,
\end{equation}
where $\sgfp$ (the asymptotic variance for the observable $f$) is given by
 \begin{equation}\label{sigmap}
    \begin{split}
      \sgfp
      &= 2 \int \big| \nabla_z \phi(z) \big|^2 \mu^P(dz)  \\
      &= 2\int \big|\nabla_z \phi(P^{1/2} x)\big|^2 \mu(dx),
   \end{split}
\end{equation}
where $\mu^P(dz) = \psi_\infty (z) dz.$
\end{theorem}
\begin{proof}
   See Section~\ref{ann:tcl}.
\end{proof}

\subsection{Explicit solution for linear observables}

\label{sec:linear_observables}

In this section we exhibit explicit solutions $\phi$ of the Poisson equation
\eqref{poisson2} when $E_N$ is quadratic and $f$ is linear, and compute the
associated asymptotic variance $\sgfp$. This simplest possible case is
of course well-known but we summarise it nevertheless to prepare for more
interesting cases.

Suppose, therefore, that
\begin{equation*}
   P = I, \quad E_N(x) = \frac12 x^T H x \quad \text{and} \quad
   f(x) = v \cdot x,
\end{equation*}
where $v \in \R^N$. From symmetry it follows that $\mu(f) = 0$,  hence the
Poisson equation \eqref{poisson2} becomes
\begin{equation}\label{appois2}
   Hx \cdot \nabla\phi(x) - \Delta\phi(x)= v \cdot x.
\end{equation}
Seeking a solution of the form $\phi(x) = d \cdot x$ with $d \in \R^N$, we
obtain $d = H^{-1} v$, i.e.,
$$
   \phi(x) = x \cdot H^{-1} v,
$$
and in particular,
\begin{equation*}
   \sgf = \sigma_{f,I}^2 = 2\int |H^{-1} v|^2 \mu(dx) = 2|H^{-1} v|^2.
\end{equation*}
In particular, choosing $v$ to be a normalised eigenmode of $H$ with associated
eigenvalue $\lambda$, we obtain $\sgf = 2\lambda^{-2}$.

Focusing specifically on a $d$-dimensional lattice model, we know from
\eqref{eq:spectrum-dimension} that $\min \sigma(H) \sim N^{-2/d}$ while
$\max\sigma(H) \sim 1$. Hence, it follows that $\sgf$ is moderate for the
high-frequency eigenmodes, but large for the observables corresponding to
low-frequency eigenmodes.

Next, we turn to the preconditioned dynamics. In this case we effectively
replace $H$ with $H_P = P^{-1/2} H P^{-1/2}$ and $v$ with $P^{-1/2} v$
and thus obtain
\begin{equation} \label{eq:sigfp-linear-v1}
   \sgfp = 2\int |H_P^{-1} P^{-1/2} v|^2 \mu(dx)
         = 2| P^{1/2} H^{-1} v |^2 =: 2|H^{-1} v|_P^2.
\end{equation}
If we assume that $c_0 = \min \sigma(H_P), c_1 = \max \sigma(H_P)$, then a
simple rewrite yields
\begin{equation} \label{eq:sgfp-bounds-Hinv}
   2c_0 v^T H^{-1} v \leq \sgfp \leq 2c_1 v^T H^{-1} v.
\end{equation}
Comparing with $\sgf =2v^T H^{-2} v$ and recalling our standing assumption
\eqref{eq:bounds:HP} that  $c_0 \sim 1, c_1 \sim 1$ (the spectrum is bounded
above and below independently of $N$), we conclude that preconditioning does not
entirely remove ill-conditioning but it is (potentially) diminished.

More concretely, for a $d$-dimensional lattice system, we obtain that
$$
     N^{-2/d} \lesssim \frac{\sgfp}{\sgf} \lesssim 1,
$$
and both bounds are attained for specific observables. We conclude that
preconditioned Langevin {\em can} be significantly more efficient than standard
Langevin (low-frequency observables) but that it will be {\em comparable in
efficiency} for high-frequency observables.

Intuitively, low-frequency observables are ``macroscopic'' in nature and include
e.g. energy, average bond-length etc., while high-frequency observables
include in particular single bonds, bond angles and dihedral angles (in a large molecule) or
a bond-length near a crack-tip. In the next sections, we consider three toy
models mimicking ``realistic'' observables of these kinds, occurring
in real-world simulations, to further substantiate our remarks.

\subsection{Example 1: Energy per particle}
\label{sec:energy}
We now consider $E_N(x) = \frac12 x^T H x$ and $f(x) = N^{-1} E_N(x)$. A straightforward
computation yields
\begin{equation}
   \label{lem-1}
   \< x \cdot Bx \>_\mu
   = \tr(H^{-1} B) \qquad \text{for } B \in \R^{N \times N},
\end{equation}
which in particular implies that
$$
   \mu(f) = \frac{1}{N} \frac{\int E_N(x) e^{- E_N(x)}}{ \int e^{-E_N(x)}}
         = \frac{\tr I}{2N} = \frac{1}{2}.
$$
Thus, the Poisson equation becomes
$$
   H x \cdot \nabla \phi(x) - \Delta \phi(x)
      = {\textstyle \frac{1}{2N}} x^T H x - {\textstyle \frac{1}{2}}
$$
We seek a solution of the form $\phi(x) = \frac12 x^T B x + l \cdot x -\tr B$ (to ensure that $\mu(\phi) = 0$), then
this yields the equation
$$
   Hx \cdot (Bx + l) - \tr B = {\textstyle \frac1{2N}} x^T H x - {\textstyle \frac{1}{2}}.
$$
This is satisfied for $B = \frac1{N} I, l = 0$, hence $\phi(x) = \frac1{2N} |x|^2 -\frac{1}{2}$.

We can now compute the asymptotic variance as
$$
   \sgf = \sigma_{f,I}^2 =2 \int \big|{\textstyle \frac1{N}} x\big|^2 \mu(dx)
      = {\textstyle \frac{2}{N^2}} \tr H^{-1}.
$$

Repeating the same argument in transformed coordinates $z = P^{1/2} x$, we also obtain the asymptotic
variance of the energy for the preconditioned Langevin dynamics:
$$
   \sgfp = {\textstyle \frac2{N^2}}  \tr H_P^{-1}
         = {\textstyle \frac2{N^2}} \tr \big( H^{-1} P \big).
$$

Let us now focus on a lattice model, where we have \eqref{eq:spectrum-dimension}.
Then we obtain that
$$
  \frac{1}{2} \sgf \approx N^{-2} \sum_{j = 1}^N (j/N)^{-2/d}
      \approx N^{-1} \int_{1/N}^1 s^{-2/d} ds \approx
      \begin{cases}
         1, & d = 1, \\
         N^{-1} \log N, & d = 2, \\
         N^{-1}, & d = 3,
      \end{cases}
$$
while, clearly, $\frac{1}{2}\sgfp \approx N^{-1}$. In summary,
$$
   \frac{\sgf}{\sgfp} \approx \begin{cases}
      N, & d = 1, \\
      \log N, & d = 2, \\
      1, & d = 3;
   \end{cases}
$$
that is, preconditioning only gives a significant speed-up in one-dimensional
lattices but not in two- or three-dimensional lattices.

\subsection{Example 2: Bond-length}
\label{sec:bondlength}
In our second example we observe a single bond in the crystal or
molecule. That is, we still use $E_N(x) = \frac12 x^T H x$ but the observable
is now given by
$$
   f(x) = x_{i} - x_j \qquad \text{for some fixed bond $i \sim j$}.
$$
This is a linear observable, hence a special case of the discussion in
Section~\ref{sec:linear_observables}. Hence, we obtain
$$
   \sgf = 2|H^{-1} l|^2 \qquad \text{where} \quad l_n =
   \begin{cases}
      1, & n = i \\
      -1, & n = j, \\
      0, & \text{otherwise.}
   \end{cases}
$$

In order to estimate $\sgf$ further we consider again the $d$-dimensional
lattice model \eqref{eq:quad-lattice-model} and $P$ given by
\eqref{eq:precond_nonlinear}. For $d \geq 2$, since $P$ is a homogeneous discrete
elliptic operator, we know from [\refcite{olson}] that
$$
   \big| [P^{-1}]_{ni} - [P^{-1}]_{nj} \big| \leq C (1 + |n-i|)^{-d},
$$
where we note that $i,j$ are now neighbouring lattice sites; i.e.,
$[P^{-1}]_{ni} - [P^{-1}]_{nj}$ denotes a discrete gradient of the lattice
Green's function.

Therefore, we obtain that
\begin{equation} \label{eq:sgf-bound-bonglength}
   \sgf = 2|H^{-1} l|^2 \lesssim 2|P^{-1} l|^2
    \lesssim 2\sum_{n \in \mathbb{Z}^d} (1+|n-i|)^{-2d} < \infty;
\end{equation}
that is, $\sgf$ has an upper bound that is independent of $N$. A
lower bound follows simply from the fact that $\|H\| \leq 1$ and hence
$v^T H^{-1} v \geq |v|^2$, which implies
$$
   \sgf \geq |l|^2 = 1.
$$

To obtain bounds on $\sgfp$, we use
\eqref{eq:sgfp-bounds-Hinv} to estimate
$$
   2c_0 = 2c_0 |l|^2 \leq \sgfp \leq 2c_1 l^T H^{-1}  \leq
   2c_1 |l| |H^{-1} l|,
$$
and we have already shown in \eqref{eq:sgf-bound-bonglength} that this is
bounded above, independently of $N$.

In summary, we obtain that
$$
   \frac{\sgf}{\sgfp} \lesssim 1 \qquad \text{for } d \geq 2,
$$
that is, we expect no substantial (if any) speed-up for the bond-length
observable from preconditioning for $d \geq 2$.

By contrast, for $d = 1$, the system $P^{-1} l$ can be solved explicitly,
and in this case one obtains $\sgf \sim N$ as $N \to \infty$ (specifically, $|P^{-1} l|^2 = N/12$), that is,
$$
   \frac{\sgf}{\sgfp} \sim N \qquad \text{for } d = 1.
$$

Thus, we conclude that preconditioning helps to accelerate the computation of the
bond-length observable only for one-dimensional structures. The intuitive
explanation of this effect is that far-away regions of space have little
influence on a single bond and hence only local equilibration matters.
The difference in 1D is that elastic interaction is naturally more long-ranged
than in dimension $d > 1$.

\subsection{Example 3: Umbrella sampling}
\label{sec:umbrella}
Our final example is inspired by a technique called umbrella sampling [\refcite{umbrella}].
Given a potential energy $E_N(x)$ and a reaction coordinate $\xi(x)$, we wish
to compute
$$
   A(\xi_0) := - \log \int e^{-E_N(x)} \delta(\xi(x)-\xi_0) \,dx.
$$
Umbrella sampling achieves this by placing a restraint on the potential energy,
$$
   E_{N,K}(x) := E_N(x) + {\textstyle \frac{K}2} \big( \xi(x) - \xi_0 \big)^2,
$$
for some $K > 0$, $\xi_0 \in \R$.
Let $\mu_{K}$ denote the corresponding equilibrium measure,
then it can be shown [\refcite{umbrella}] that, for $K$ large,
$$
   \frac{\partial A}{\partial \xi}\Big|_{\xi = \bar{\xi}_0}
   \approx - K (\bar{\xi}_0  - \xi_0) \qquad \text{where }
   \bar{\xi}_0 = \< \xi \>_{\mu_K}.
$$
Thus, $\partial_\xi A$ and hence $A$ can be reconstructed in this way.
More sophisticated variations of the idea exist of course, but for the sake
of simplicity of presentation we will focus on this particularly simple
variant.

To construct an analytically accessible toy problem mimicking umbrella
sampling we consider again a quadratic energy $E_N(x) = \frac12 x^T H x$
and a linear reaction coordinate $\xi(x) := l \cdot x$ where, for simplicity,
we assume that $|l| = 1$ (for $|l| = O(1)$, the argument is analogous).
The restrained potential with penalty parameter $K > 0$  is then
given by
$$
   E_{N,K}(x) = {\textstyle \frac12} x^T H x
            + {\textstyle \frac{K}2} (\xi(x) - \xi_0)^2,
$$
for some $\xi_0 \in \R$, while the observable from which we can reconstruct the mean force is
simply
$$
   f(x) = K \big(l \cdot x - \xi_0\big).
$$
We are again in the context of Section~\ref{sec:linear_observables}
and therefore obtain
$$
   \sgf = 2K^2 \big| H_K^{-1} l \big|^2,
$$
where
$$
   H_K = \nabla^2 E_{N,K} = H + K l l^T.
$$
The Sherman--Morrison formula yields
\begin{equation} \label{eq:sherman-morrison}
   H_K^{-1} l = \bigg( H^{-1} - \frac{K H^{-1} l l^T H^{-1}}{1 + K l^T H^{-1} l} \bigg) l =  \frac{H^{-1} l}{1 + K l^T H^{-1} l},
\end{equation}
and hence,
\begin{align*}
   \sgf = 2K^2 \big|H_K^{-1}l\big|^2
      = \frac{2K^2 |H^{-1} l|^2}{(1 + K l^T H^{-1} l)^2}.
\end{align*}

Since our focus in the present example is the ill-conditioning induced by large
$K$ rather than ill-conditioning induced by $H$ (e.g. through system size $N$),
let us assume that $\max \sigma(H) = 1$ (as always) while $\min \sigma (H) \geq
c_0$, for some moderate constant $c_0$. This would, e.g., be the typical
situation for a small molecule, or if we preconditioned $H$ but without
accounting for the umbrella. We then obtain
\begin{align*}
   c_0^2 \leq \frac{2K^2 c_0^2}{(1 + K)^2} \leq \sgf \leq \frac{2K^2}{(1+ K c_0)^2} \leq \frac{2}{c_0^2};
\end{align*}
that is, $\frac{1}2\sgf \sim 1$ as $K \to \infty$.

By contrast, suppose now that we choose a preconditioner
$$
   P_K := P + K l l^T,
$$
where $P$ is a preconditioner for $H$ satisfying $c_0^P v^T P v \leq v^THv \leq v^T
P v$. Then a straightforward calculation yields
$$
   c_0^P \leq
   \frac{v^T H v + K (l \cdot v)^2}{v^T P v + K (l \cdot v)^2}
   = \frac{v^T H_K v}{v^T P_K v} \leq 1.
$$
It follows from \eqref{eq:sgfp-bounds-Hinv} that
$$
   \sgfp = 2K^2 \big| H_K^{-1} l \big|_{P_K}^2
   \approx 2K^2 |H_K^{-1/2} l|^2,
$$
where $\approx$ now indicates upper and lower bounds with constants
that are independent of $K$.

Using \eqref{eq:sherman-morrison} we obtain
$$
   \sgfp \approx 2K^2 \frac{l^T H^{-1} l}{1 + K l^T H^{-1} l} \sim K
   \qquad \text{as $K \to \infty$.}
$$

We can therefore conclude that
$$
   \frac{\sgfp}{\sgf} \approx
    \frac{(l^T H^{-1} l)(1 + K l^T H^{-1} l)}{|H^{-1} l|^2} \sim 1+K
    \qquad \text{as $K \to \infty$;}
$$
that is, preconditioning the umbrella  actually achieves a significant {\em
deterioration} of the asymptotic variance and thus the $P$-Langevin dynamics
actually becomes {\em less efficient} than the standard Langevin dynamics.
However note the following crucial remark:

\begin{remark}[Step-sizes revisited] \label{rem:stepsizes}
   The surprising result of the present section does not in fact fully
   fall within our starting assumptions. While $\|H\| = 1$, $\|H_K\|$ is in fact
   of order $O(1 + K)$ which means that the time-step for the discretisation of
   the Langevin equation should be of order $O(K^{-1})$, which exactly
   balances the lower mixing of the preconditioned dynamics and make the
   two schemes again comparable.  Indeed, in our computational examples
   we will need to choose $\Delta t = O(K^{-1})$ to prevent instability.

   In practice, the restraint parameter $K$ is chosen of
   the same order of magnitude of the stiffest bond in a molecule, while the
   reaction coordinate will normally be a function of the softest bonds and hence it
   would create no additional time-step restriction. In such a situation,
   it is indeed preferable to {\em not} precondition the restraint.

    However, we emphasise again that the interaction between preconditioning and
    time-stepping is an issue that we do not properly address in the present
    work and which will require further attention in the future.
\end{remark}

\section{Numerical Tests}\label{test}
\label{sec:numtests}
We conclude our discussion by demonstrating the extension of our
explicit computations to  a mildly non-linear lattice model. As potential
energy $E_{M^d} : \R^{\{1,\dots,M\}^d} \to \R$, we choose
$$
   E_N(u) := \sum_{\alpha_1=0}^M \cdots \sum_{\alpha_d = 0}^M
         \sum_{j = 1}^d  \phi\Big(\smfrac{1}{\sqrt{d}} \big(u(\alpha+e_j) - u(\alpha)\big)\Big),
$$
where $u(\alpha) := 0$ if any $\alpha_j \in \{0, M+1\}$, and
with convex nearest-neighbour pair potential
$$
   \phi(r) = {\textstyle \frac18} \big( r^2 + \sin(r)^2 \big).
$$
Upon choosing an arbitrary linear labelling of indices $\alpha \in \{1,\dots,M\}^d$,
this is a special case of~\eqref{eq:nonlinear_model}. As preconditioner we choose
\eqref{eq:precond_nonlinear}, which we can also write as
$$
   \< Pv, v \> = \frac1{4 d} \sum_{\alpha_1=0}^M \cdots \sum_{\alpha_d = 0}^M
         \sum_{j = 1}^d |u(\alpha+e_j) - u(\alpha)|^2.
$$
The occurrence of $d$ in the definitions of $E_N$ and $P$ ensures that $\| \nabla E_N\| \approx  \|P \| \approx 1$; moreover, since $\phi$ is strictly convex, we have
that ${\rm cond}((\nabla E_N(x))_{P})$ is bounded above independently of $N$;
cf. \eqref{eq:bounds:HP}.

For all simulations (with small modifications for the umbrella sampling
example) we choose
$$
   \delta = 0.1, \quad N_{\rm steps} = 10^5, \quad N_{\rm runs} = 400.
$$
We then use a Cholesky factorisation to compute $P$, i.e. $P = L L^T$, followed by
$$
   X_{n+1}^P = X_n^P - \delta P^{-1} \nabla E_N(X_N^P) + \sqrt{2 \delta} L^{-T} R_n,
   \quad \text{for } n = 1, \dots, N_{\rm steps},
$$
where $R_n \sim N(0, I)^N$. The estimate of the observable value is then given by
$\bar{f} = N_{\rm steps}^{-1} \sum_{n = 1}^{N_{\rm steps}} f(X_n^P)$.
We compute $N_{\rm runs}$ trajectories in order to estimate the
asymptotic variance from $N_{\rm runs}$ independent samples of $\bar{f}$.

For the umbrella sampling test, we choose $\xi(x) := l \cdot x$ to be the
bond-length observable again with $\xi_0 = 0.33$. The modification to the
energy and observable is then as described in Section \ref{sec:umbrella}.
With $\delta = 0.1$, the discretised Langevin dynamics turns out to be
unstable, hence we had to choose $\delta_K = 0.1 / K$ instead. To account for
this (see also Remark \ref{rem:stepsizes}), we study $K \sgf$ instead of $\sgf$ in
our tests. We perform the
umbrella sampling test only for $d = 2, N = 8^2$, since we focus here on
the magnitude on the restraint parameter $K$ rather than the
system size.

The results of the simulations are shown in table \ref{table}. The numbers
closely match the analytical predictions of Sections \ref{sec:energy},
\ref{sec:bondlength} and \ref{sec:umbrella}. To conclude this discussion
we only remark that we did {\em not} fine-tune step-sizes, which means one
could likely make small improvements to both the preconditioned and
unpreconditioned processes. However, we believe that the {\em trends}
across dimension, system size and restraint parameter are reliable.

In particular, we stress that even though the preconditioned variants often
have a smaller asymptotic variance,  often (in particular for $d = 3$) this
improvement is only by a moderate constant factor. Because only the trends are
reliable indicators in these tests only a successive improvement with increasing
$N$ or $K$ (e.g. as in the $d = 1$ tests) can be considered a success for the
preconditioned algorithm.

\begin{table} \centering
   {\bf Asymptotic Variance: Energy} \\[1mm]
\begin{tabular}{c|cc||c|cc||c|cc}
   \multicolumn{3}{c||}{$d = 1$} & \multicolumn{3}{c||}{$d = 2$} & \multicolumn{3}{c}{$d = 3$} \\
   $N^{\frac1d}$ & $\sgf$ & $\sgfp$ & $N^{\frac1d}$ & $\sgf$ & $\sgfp$ & $N^{\frac1d}$ & $\sgf$ & $\sgfp$ \\
   \hline
   8 & 3.7e-01 & 6.4e-02  & 4 & 1.1e-01 & 4.4e-02  & 4 & 2.8e-02 & 9.8e-03  \\
   16 & 3.6e-01 & 3.7e-02  & 8 & 3.2e-02 & 9.3e-03  & 6 & 8.0e-03 & 3.5e-03  \\
   32 & 3.0e-01 & 1.8e-02  & 16 & 1.0e-02 & 2.3e-03  & 9 & 2.3e-03 & 9.4e-04  \\
   128 & 1.7e-01 & 4.5e-03  & 32 & 2.8e-03 & 5.4e-04  & 13 & 7.9e-04 & 3.0e-04
\end{tabular}

\bigskip

{\bf Asymptotic Variance: Bond-length} \\[1mm]
\begin{tabular}{c|cc||c|cc||c|cc}
\multicolumn{3}{c||}{$d = 1$} & \multicolumn{3}{c||}{$d = 2$} & \multicolumn{3}{c}{$d = 3$} \\
$N^{\frac1d}$ & $\sgf$ & $\sgfp$ & $N^{\frac1d}$ & $\sgf$ & $\sgfp$ & $N^{\frac1d}$ & $\sgf$ & $\sgfp$ \\
\hline
8 & 2.4e+01 & 7.4e+00  & 4 & 3.2e+01 & 8.5e+00  & 4 & 1.9e+01 & 7.5e+00  \\
16 & 4.2e+01 & 7.2e+00  & 8 & 5.4e+01 & 9.5e+00  & 6 & 1.8e+01 & 7.6e+00  \\
32 & 7.6e+01 & 7.4e+00  & 16 & 7.2e+01 & 1.1e+01  & 9 & 1.9e+01 & 7.5e+00  \\
128 & 2.8e+02 & 7.0e+00  & 32 & 1.1e+02 & 1.1e+01  & 13 & 2.2e+01 & 7.7e+00
\end{tabular}

\bigskip

{\bf Asymptotic Variance: Umbrella Sampling} \\[1mm]
\begin{tabular}{c|cc}
\multicolumn{3}{c}{$d = 2, N^{1/2} = 8$} \\
$K$ & $K \sgf$ & $\sgfp$  \\
\hline
10 & 1.6e+01 & 1.8e+01  \\
20 & 3.5e+01 & 4.0e+01  \\
40 & 6.7e+01 & 7.7e+01  \\
80 & 1.3e+02 & 1.8e+02
\end{tabular}

\caption{Numerically estimated  asymptotic variances of the energy observable
(Section \ref{sec:energy}), the bond-length observable (Section
\ref{sec:bondlength}) and the restraint observable occurring in umbrella sampling
(Section \ref{sec:umbrella}). The nonlinear potential energy used in these tests
is described in Section \ref{sec:numtests}. All results match the
analytical predictions of Sections \ref{sec:energy}, \ref{sec:bondlength} and \ref{sec:umbrella}.}\label{table}
\end{table}

\section*{Conclusion}
In this paper we strived to develop an intuition what the effect of
preconditioning has on molecular simulations. The results are very mixed: it is
clear that preconditioning accelerates convergence of the probability
density functions to equilibrium (see Theorem \ref{th:exponential_convergence}
as well as the discussion in Section \ref{sec:linear_observables}), and this
necessarily implies accelerated convergence for {\em some} observables. However,
for many concrete observables of practical importance little (if any) benefit
can be gained. This was a surprising outcome for us and indicates that
alternative avenues need to be explored on how a priori information about the
analytical structure of  configuration space should be exploited in molecular
simulation.

We emphasize again that our (partially negative) conclusion, contrary
to much of the existing literature, is due to the
fact that we test the convergence of specific observables.
 Moreover, we stress that we have only performed a  limited set of tests on
 highly simplified
toy models and a limited set of observables, while more realistic models may exhibit many features that we neglected.

\section{Proofs}
\label{sec:proofs}

\subsection{Proof of Proposition \ref{th:time-discretisation}}\label{prp}

\begin{proof}
The covariance of the invariant measure associated to the dynamics \eqref {eq:pol_discrete} is given by the following identity:

$$\begin{aligned}
C^P_\delta&=(I-  \delta P^{-1} H)C^P_\delta(I-  \delta HP^{-1})+2 \delta  P^{-1}\\
C^P_\delta&=C^P_\delta- \delta  P^{-1} HC^P_\delta- \delta  C^P_\delta HP^{-1}+ \delta ^2P^{-1}HC^P_\delta HP^{-1}+2 \delta  P^{-1}\\
P^{-1}HC^P_\delta + C^P_\delta HP^{-1}&= 2 P^{-1}+ \delta  P^{-1}HC^P_\delta HP^{-1}.
\end{aligned}$$
Expanding $C^P_\delta$, we have
 $$C^P_\delta\sim C_0+ \delta  C_1+\delta ^2C_2+ O( \delta ^3) ,$$
then one gets
$$\begin{aligned}
P^{-1}HC_0 + C_0HP^{-1}=2P^{-1}&\quad\Rightarrow C_0=H^{-1}\\
P^{-1}HC_1+C_1HP^{-1}= P^{-1}HP^{-1}&\quad\Rightarrow C_1=\frac{1}{2}P^{-1}\\
P^{-1}HC_2+C_2HP^{-1}=\frac{1}{2}P^{-1}HP^{-1}HP^{-1}&\quad\Rightarrow C_2=\frac{1}{4}P^{-1}HP^{-1}.
\end{aligned}$$
Proceeding by induction, one can therefore obtain:
$$\forall k\in \mathbb{N},\quad C_k=2^{-k}(P^{-1}H)^kH^{-1}.$$
Therefore $C^P_\delta$ can be rewritten as:
$$ C^P_\delta=\displaystyle\sum_{k=0}^{+\infty}\lp {\textstyle \frac{ \delta }{2}}P^{-1}H \rp^kH^{-1}. $$
Indeed, the sum  $\displaystyle\sum_{k\geq 0}C_k$ converges since $\frac{\delta  }{2}|P^{-1}H|<1.$ The covariance operator be can rewritten as
$$\begin{aligned}
C^P_\delta & = \left( \displaystyle\sum_{k=0}^\infty \left({\textstyle \frac{ \delta }{2}}P^{-1}H\right)^k\right)H^{-1}\\
& = \left(I-{\textstyle \frac{ \delta }{2}}P^{-1}H\right)^{-1}H^{-1}\\
&= (H-{\textstyle \frac{ \delta }{2}}HP^{-1}H)^{-1},
\end{aligned}$$
which concludes the proof.
\end{proof}

\subsection{Proof of Theorem \ref{th:exponential_convergence}}\label{prtheo1}

\begin{proof}
Under the assumptions on the potential $E_N$, see  A.19
in [\refcite{Villani}], the density $\psi_\infty $ satisfies a Poincar\'e inequality: there exists
$\lambda_P>0$ such that for all probability density functions $\phi $, we have:
\begin{equation}\label{pcr}
\displaystyle\int_{\R^N}\left| \frac{\phi}{\psi_\infty }-1 \right|^2\psi_\infty dz\leq \frac{1}{\lambda_P}\int_{\R^N}\left| \nabla\left(\frac{\phi}{\psi_\infty }\right) \right|^2\psi_\infty  dz.
\end{equation}

The optimal parameter $\lambda_P$ in \eqref{pcr} is the opposite of the smallest (in absolute value) non-zero eigenvalue of the Fokker-Planck operator $\mathcal{L}_P^* $, which is self-adjoint in $L^2(\R^N,\psi_{\infty}^{-1}dx)$.
 Thus the exponent $\lambda_P$ is the spectral gap of $\mathcal{L}_P^* $.

If $\psi_t^P$ is a solution of \eqref{fpfp}, therefore for all initial condition $\psi_0\in L^2(\psi_{\infty}^{-1})$, $\forall t\geq 0:$
\begin{equation}\label{eqm}
\displaystyle \frac{d}{dt}\|\psi_t^P-\psi_\infty \|_{L^2(\psi_{\infty}^{-1})}^2=-2\int_{\R^N}\left|\nabla\lp\frac{\psi_t^P}{\psi_\infty }\rp \right|^2\psi_\infty  dz.
\end{equation}
Indeed,
$$\begin{aligned}
\displaystyle \frac{d}{dt}\|\psi_t^P-\psi_\infty \|_{L^2(\psi_{\infty}^{-1})}^2&=\displaystyle \frac{d}{dt}\int \left |\psi_t^P-\psi_\infty \right |^2\psi_{\infty}^{-1}dz\\
&=2\int\pr_t\psi_t^P\lp\psi_t^P-\psi_\infty \rp\psi_{\infty}^{-1}dz\\
&=2\int \nabla\cdot\lp \nabla E_N^P(z)\psi_t^P+\nabla\psi_t^P\rp\lp\frac{\psi_t^P}{\psi_\infty }-1\rp dz\\
&=-2\int \lp \nabla E_N^P(z)\psi_t^P+\nabla\psi_t^P\rp\cdot \nabla\lp\frac{\psi_t^P}{\psi_\infty }\rp dz.
\end{aligned}$$
But we have
$$\begin{aligned}
\nabla E_N^P(z)\psi_t^P+\nabla\psi_t^P&=-\nabla \lp\ln \lp\psi_\infty \rp\rp\psi_t^P+\nabla\psi_t^P\\
&=-\frac{\nabla\psi_\infty \psi_t^P}{\psi_\infty }+\nabla\psi_t^P\\
&=\nabla\lp\frac{\psi_t^P}{\psi_\infty }\rp\psi_\infty ,
\end{aligned}$$
which yields \eqref{eqm}. Therefore, using \eqref{pcr},
$$\begin{aligned}
\displaystyle \frac{d}{dt}\|\psi_t^P-\psi_\infty \|_{L^2(\psi_{\infty}^{-1})}^2&\leq -2\lambda_P \int \left| \frac{\psi_t^P}{\psi_\infty }-1 \right|^2\psi_\infty  dz\\
&= -2\lambda_P \int \left| \psi_t^P-\psi_\infty \right|^2\psi_\infty^{-1} dz,
\end{aligned}$$
then
$$\|\psi_t^P-\psi_\infty \|_{L^2(\psi_{\infty}^{-1})}^2\leq \e^{-\lambda_P }\|\psi_0-\psi_\infty \|_{L^2(\psi_{\infty}^{-1})}^2.$$

\end{proof}

\subsection{Preliminaries for spectral analysis}

In the linear case, i.e $E_N^P(z)=\frac{1}{2}z^TH_Pz$, where $H$ and $P$ are symmetric positive definite and $H_P := P^{-1/2} H P^{-1/2}$, the analysis will be carried out in a suitable system of coordinates which simplifies the calculations and the proofs of the main theorems. For this reason, we will perform one conjugation and one additional change of variables .

From the partial differential equation point of view and in order to use standard techniques from the spectral analysis of partial differential equations, then it appears to be useful to work in $L^2(\R^N,dz;\C) $ instead of $L^2(\R^N,\psi_\infty  dz;\C) $.  The mapping $\phi\mapsto \psi_\infty ^{-1/2}\phi $ maps unitarily $L^2(\R^N,dz;\C) $  into $L^2(\R^N,\psi_\infty  dz;\C) $ with the associated transformation rules for the differential operators:
$$\e^{-\frac{1}{2}H_P}\nabla_z \e^{\frac{1}{2}H_P}=\nabla_z +\frac{1}{2}\nabla_z H_P.$$
Thus, the operator $\mathcal{L}_P= -\nabla_z E_N^P(z)\cdot\nabla_z +\Delta_z$ is transformed into
\begin{align}
\displaystyle \overline{\mathcal{L}}_P&=\e^{-\frac{1}{2}H_P}\mathcal{L}_P\e^{\frac{1}{2}H_P}\nonumber\\
&= \Delta_z-\frac{1}{4}|\nabla_z E_N^P(z)|^2+\frac{1}{2}\Delta_z E_N^P(z)\nonumber\\
&= \Delta_z-\frac{1}{4}z^T H_P^2z+\frac{1}{2}\tr(H_P).
\end{align}

The kernel of $ \overline{\mathcal{L}}_P$ is $\C\e^{-\frac{z^TH_Pz}{4}} $ and the operator $\overline{\mathcal{L}}_P$ is unitarily equivalent to the operator $\mathcal{L}_P$.

In the goal of modifying the kernel of the operator $ \overline{\mathcal{L}}_P$ into a centered Gaussian with identity covariance matrix, we perform a second change of variables. In the following, we introduce the new coordinates $y = H_P^{1/2}z$, so that $\nabla_z=H_P^{1/2}\nabla_y$. Then the operator $ \overline{\mathcal{L}}_P$ becomes:
\begin{equation}
\displaystyle \tilde{\mathcal{L}}_P=\nabla_y^TH_P\nabla_y-\frac{1}{4}y^TH_Py+\frac{1}{2}\tr(H_P).
\end{equation}

The operator $\tilde{\mathcal{L}}_P$ is still acting in $L^2(\R^N,dz;\C) $. In the new coordinate system ($Y_t=H_P^{1/2}Z_t$), the corresponding stochastic process is:
$$ dY_t=-H_PY_tdt+\sqrt{2}H_P^{1/2}dW_t, $$
 so that ${\rm Ker}(\tilde{\mathcal{L}}_P)=\frac{1}{(2\pi)^{N/4}}\e^{\frac{|y|^2}{4}}.$ The last conjugation and change of variables are used to compute the spectrum of $\mathcal{L}_P^{*}$ needed to proof Theorem~\ref{spec} (see Section~\ref{prtheo2}).

Let us now introduce some additional notations. Recall the space of rapidly decaying complex valued $\mathcal{C}^{\infty}$ functions
$$\mathcal{S}(\R^N)=\displaystyle \left\{ f\in \mathcal{C}^{\infty}(\R^N),\forall\alpha,\beta\in \mathbb{N}^{N},\exists C_{\alpha,\beta}\in \R_+, \sup_{x\in\R^N} |x^\alpha\partial_x^\beta f(x) |\leq C_{\alpha,\beta} \right\}, $$
and its dual is denoted $\mathcal{S}'(\R^N)$.

The Weyl-quantization $q^W(x,D_x) $ of a symbol $q(x,\xi)\in \mathcal{S}'(\R^N) $ is an operator defined by its Schwartz-kernel
$$\displaystyle \left[q^W(x,D_x)\right](x,y)=\int_{\R^N}\e^{i(x-y)\cdot\xi }q\left( \frac{x+y}{2},\xi \right)\frac{d\xi}{(2\pi)^N}. $$
For instance, the Weyl symbol of of the operator
$$-\tilde{\mathcal{L}}_P+\frac{1}{2}\tr(H_P)= -\nabla_y^TH_P\nabla_y+\frac{1}{4}y^TH_Py$$
is
\begin{equation}\label{wey}
q(y,\xi)=\xi^TH_P\xi+\frac{y^TH_Py}{4}.
\end{equation}
Those tools are essential to proof Theorem~\ref{spec} (see Section~\ref{prtheo2}). For more details on Weyl-quantization, one can refer to [\refcite{ott}].

\subsection{Proof of Theorem \ref{spec}}\label{prtheo2}

\begin{proof}
Referring to Theorem~1.2.2 in [\refcite{hit}], the spectrum of the operator $ q^W(y,D_y)=-\tilde{\mathcal{L}}_P+\frac{1}{2}\tr(H_P)$ associated with the elliptic
quadratic Weyl symbol $q(y,\xi) $ defined by \eqref{wey} is given by
$$\sigma(q^W(y,D_y))=\displaystyle\left\{ \sum_{\substack{\lambda\in \sigma(G)\\ {\rm Im}\lambda\geq 0}}-i\lambda( r_\lambda+2k_\lambda), k_\lambda\in \mathbb{N}\right\}.
 $$
where $G$ is the so-called Hamilton map associated with $q$ , and $r_\lambda$ is the algebraic multiplicity of $\lambda\in \sigma(G)$ (the dimension of the characteristic space). The Hamilton map is the $\C$-linear map $G: \C^{2N}\rightarrow \C^{2N} $ associated with the matrix

$$G=\begin{bmatrix}
 0 & H_P   \\
 -\frac{1}{4}H_P& 0 \\
\end{bmatrix}\in \C^{2N\times 2N}. $$
 The matrix $G$ is similar to another matrix denoted $\overline{G}$ and defined by
$$\overline{G}=\begin{bmatrix}
 \frac{1}{\sqrt{2}} & 0   \\
 0& \sqrt{2} \\
\end{bmatrix}G\begin{bmatrix}
\sqrt{2}& 0   \\
 0&  \frac{1}{\sqrt{2}}  \\
\end{bmatrix}=\frac{1}{2} \begin{bmatrix}
 0 & H_P  \\
 -H_P& 0 \\
\end{bmatrix}. $$
Now, the characteristic polynomial of $G$ can be computed by\\
$$\begin{aligned}
\dtr (G-\lambda I)&=\dtr({\overline{G}-\lambda I })=2^{-2N}\left |\begin{array}{cc}
-2\lambda I &H_P\\
-H_P&-2\lambda I
\end{array}
\right |\\
&=2^{-2N}\left |\begin{array}{cc}
-2\lambda I &H_P\\
-H_P-i2\lambda I& i(H_P+i2\lambda I)
\end{array}
\right |\\
&=2^{-2N}\left |\begin{array}{cc}
i(-H_P+i2\lambda I ) &H_P\\
0& i(H_P+i2\lambda I)
\end{array}
\right |\\
&=2^{-2N}\dtr(H_P-i2\lambda I)\dtr(H_P+i2\lambda I).
\end{aligned}$$
Since ${\rm Re}(\sigma(H_P) )\geq 0$, one thus obtains that
$$\sigma(G)\cap \{\lambda, {\rm Im}\lambda\geq 0 \}= \frac{i}{2}\sigma(H_P). $$
In particular,
$$\displaystyle\sum_{\substack{\lambda\in \sigma(G)\\ {\rm Im}\lambda\geq 0}}-i\lambda 2k_\lambda=\sum_{\mu\in\sigma(H_P)}k_{\frac{i}{2}\mu}\mu $$
and
$$\sum_{\substack{\lambda\in \sigma(G)\\ {\rm Im}\lambda\geq 0}}-i\lambda r_\lambda=\frac{\tr(H_P)}{2}, $$
which concludes the proof of the theorem.
\end{proof}

\subsection{Proof of Theorem~\ref{th:poisson}}\label{prfpoi}

\begin{definition} (Foster-Lyapunov criterion)\label{def}\\
We say that the Foster-Lyapunov criterion holds for \eqref{pmala2} if there exists a function $U:\R^N \rightarrow \R $ and constants $C>0$ and $b\in \R$ such that $\mu(U)<\infty$,
\begin{equation}\label{foly}
 \mathcal{L}_PU(z) \leq -cU(z)+b\mathbf{1}_{C}
\end{equation}
and $U(z)\geq 1,\,z\in \R^N$, where $\mathbf{1}_{C}$ is the indication function over a petite Borel subset $C$ of $\R^N$ (refer to [\refcite{twe}] for more details).
\end{definition}
For the generator $\mathcal{L}_P$ corresponding to \eqref{ops}, compact sets are always petite.
In the following we prove that the Foster-Lyapunov criterion holds for \eqref{pmala2}. But first we need an assumption on the potential $E_N^P$.
\medskip

\textbf{Assumption A} There exists $k>0$ such that $\mu^P(z)$ is bounded from above for all $|z|\geq k$ and, for some $0<\beta<1$,
$$\displaystyle\lim_{|z|\rightarrow +\infty }\inf \left[(1-\beta)|\nabla E_N^P(z)|^2+\Delta E_N^P(z)\right]>0. $$

\begin{lemma}
Under Assumption A, the Foster-Lyapunov criterion holds for \eqref{pmala2} with:
\begin{equation}
U(z)=\e^{\beta E_N^P(z)}, \, 0<\beta<1.
\end{equation}
\end{lemma}

\begin{proof}
Recall the generator of \eqref{pmala2}
$$\mathcal{L}_P=-\nabla_z E_N^P(z)\nabla + \Delta.$$
For $U(z)=\e^{\beta E_N^P(z)}$, one obtains:
\begin{align}
\mathcal{L}_PU(z)&= -\nabla U(z)\cdot \nabla E_N^P(z) + \Delta U(z) \nonumber\\
&=-\nabla\e^{\beta E_N^P(z)}\cdot \nabla E_N^P(z)+\nabla\cdot(\nabla \e^{\beta E_N^P(z)} )\nonumber\\
&= -\beta |\nabla E_N^P(z)|^2 \e^{\beta E_N^P(z)}+\beta \Delta E_N^P(z)\e^{\beta E_N^P(z)}\nonumber\\
&\quad +\beta^2 |\nabla E_N^P(z)|^2 \e^{\beta E_N^P(z)}\nonumber\\
&= -\beta \left[ (1-\beta)|\nabla E_N^P(z)|^2+\Delta E_N^P(z) \right] U(z).
\end{align}
Therefore, by Assumption A, for $\varepsilon>0 $, $\exists k>0$ such that $\forall\, |z|>k$:
$$(1-\beta)|\nabla E_N^P(z)|^2+\Delta E_N^P(z)>\varepsilon, $$
and so also,
$$\mathcal{L}_PU(z)\leq -\beta \varepsilon U(z)+b\mathbf{1}_{C_k}, $$
where $C_k=\{z\in \R^N;|z|\leq k\}$ and $b>0$.

Finally, since $\psi_\infty (z)$ is bounded, then $U(z)$ is bounded away from zero uniformly. Then $U(z)$ can be rescaled to satisfy the condition $U(z)\geq 1$. Thus, we valid the Foster-Lyapunov criterion for \eqref{pmala2}.
\end{proof}

Using this lemma, the proof of the well-posedness result for the Poisson equation \eqref{poisson2} come straightforward using  Theorem~3.2 in [\refcite{bhatt}].
\subsection{Proof of Theorem~\ref{thtcl}}\label{ann:tcl}
\begin{proof}
We start the proof by decomposing $\mu_t^P(f)-\mu(f) $ into a martingale and a remainder terms:

Using \eqref{poisson2}, \eqref{ops} and \eqref{pmala2}
\begin{align}
\epsilon_t^P(f)-\mu(f) &=\displaystyle \frac{1}{t}\int_0^tf(P^{-1/2}Z_s)ds -\mu(f) \nonumber \\
&=  \displaystyle \frac{1}{t}\int_0^t\lp f(P^{-1/2}Z_s) -\mu(f)\rp ds \nonumber \\
&= \displaystyle \frac{1}{t}\int_0^t-\mathcal{L}_P\phi(Z_s)ds \nonumber \\
& = \displaystyle\frac{\phi(Z_0)-\phi(Z_t)}{t}+\frac{\sqrt{2}}{t}\int_0^t\nabla \phi(Z_s)dW_s\nonumber \\
&:= R_t+M_t.\nonumber
\end{align}

Consider now the rescaling $\sqrt{t}(\epsilon_t^P(f)-\mu(f) ) $. Using the central limit theorem for the martingale term $\sqrt{t}M_t $ (see [\refcite{hella}] , Theorem 5.3), one obtains the following convergence in distribution

$$\displaystyle \sqrt{t}M_t \xrightarrow{D} \mathcal{N}(0,\sgfp), $$
with
\begin{equation*}
\sgfp = \displaystyle 2\int |\nabla_z\phi (z) |^2\mu^P(dz)=2\int |\nabla_z\phi (z) |^2\mu(dx).
\end{equation*}
It remains to study the remainder term $\sqrt{t}R_t $. We consider the two cases: If $Z_0\sim\mu $, then since $\phi\in L^2(\mu) $, we have that
$$\displaystyle \sqrt{t}R_t \xrightarrow{D} 0 \text{ in } L^2(\mu). $$
In the more general case, we must refer to a "propagation of chaos" argument (see for example [\refcite{catt}], Section~8 and [\refcite{dunc}]), to obtain the same result.
\end{proof}

\section*{Acknowledgments}
We thank Tony Leli\`evre and Jonathan Weare for helpful discussions.
The work was supported by ERC Starting Grant 335120.

\clearpage

\appendix

\section{Setup of the Metadynamics example}
\label{sec:mtd}
%
In the example shown in Figure~\ref{fig:pmf_mtd_mtd}, we investigated
how preconditioning the {\em second order Langevin dynamics} affects the sampling
of different observables when applied to adaptive potential of mean force (PMF) techniques
such metadynamics (MTD) [\refcite{Laio02}]. Unlike the previously discussed umbrella sampling
method, adaptive PMF techniques build their bias and mean force or potential of
mean force estimates on-the-fly during the dynamics. In a general form, the
corresponding biased and preconditioned equations of motion with preconditioner
$P = P(q)$ can be written as
\begin{subequations}
	\label{eq:biased_langevin}
	\begin{align}
		dq_{t} & = M^{-1} p_{t} dt \\
		dp_{t} & = -\nabla E_{N}(q_{t}) dt + \nabla \xi(q_{t}) F_{\mathrm{b}}(\xi(q_{t}), t) - P M^{-1} p_{t} dt + \sqrt{\frac{2}{\beta}} P^{1/2} dW_t
	\end{align}
\end{subequations}
where $M \in \mathbb{R}^{N \times N}$ is the diagonal mass matrix, $\beta$ is
the inverse temperature, $\xi$ is the collective variable (reaction coordinate) and
$F_{\mathrm{b}}(\xi, t)$ is the biasing force. In the case of metadynamics,
$F_{\mathrm{b}}^{\mathrm{MTD}}(\xi, t) = -\nabla_{\xi}
E_{\mathrm{b}}^{\mathrm{MTD}}(\xi, t)$, where
$E_{\mathrm{b}}^{\mathrm{MTD}}(\xi, t)$ is some history dependent biasing
potential composed by Gaussians regularly deposited in the collective variable
space:
\begin{equation}
	\label{eq:mtd}
	E_{\mathrm{b}}^{\mathrm{MTD}}(\xi, t) = \sum \limits_{t' < t} \delta_{t'} \exp \left(-\frac{1}{2} \left(\xi - \xi(q_{t'}) \right)^{T} w^{-2} \left(\xi - \xi(q_{t'}) \right) \right)
\end{equation}
where $\delta_{t'}$ is the height of the corresponding Gaussian and $w$ is a diagonal matrix including the widths of the collective variable components.

The molecular system we chose for this test was the 2-(formylamino)\!
propionaldehyde in gas phase using the Amber99SB~[\refcite{Hornak06}] force
field. 

The molecule has a single slow degree of freedom, a dihedral angle, that
was selected as the one-dimensional collective variable (we note that this
variable is highly associated to one of the two dihedrals of alanine dipeptide,
a test system widely used in the computational chemistry field).

We performed 200 ps long molecular dynamics simulations at 300 K using the BBK
integrator scheme [\refcite{Brunger84}] with 0.5 fs timestep. In the case of
unpreconditioned dynamics $P = \gamma M$ was used with $\gamma = 5.0$ ps$^{-1}$,
while for the preconditioned dynamics we applied $P = \gamma M + \tau \tilde{H}$, where 
$\tilde{H}$ is a Hessian-based preconditioner, whose positiveness is guaranteed by 
rebuilding the matrix using the spectral decomposition of the Hessian with the absolute
values of the eigenvalues [\refcite{MonesPRE}].
We used $\gamma M$ as diagonal a stabiliser and varied $\gamma$ and $\tau$ parameters.

A deposition frequency of 1/50 fs$^{-1}$ and a starting
height of $\delta = 0.004$ eV were used in a well-tempered
variant of MTD [\refcite{Barducci08}] with $T_{\mathrm{w}} = 10000$ K. 

Free energy profiles were reconstructed simply as the negative of the actual history
dependent biasing potential. The reference for computing the RMS error of the
profiles was obtained from a 200 ns long unpreconditioned constrained dynamics
simulation [\refcite{Mones16}].

The result of this test is shown in Figure~\ref{fig:pmf_mtd_mtd},
where we plotted the RMS error of the reconstructed profiles based on 10 
independent MTD simulations for each parameter set. We observe that 
in general, preconditioning does not improve the convergence.

As a matter of fact, we performed a variety of similar tests, e.g. with
different parameters, or different observables, all of which led to
similar conclusions. Indeed, depending on the choice of observable,
preconditioning often has an even larger {\em negative} impact.


\end{document}